\documentclass[12pt]{amsart}

\usepackage{amsmath,amssymb,mathrsfs}

\usepackage{amsthm} 

\usepackage{mathrsfs,amsopn}
\usepackage[margin=3.1cm]{geometry}

 \usepackage[dvips]{graphicx,color}

\usepackage{amsfonts}
\usepackage{eucal}
\usepackage{times}

\numberwithin{equation}{section}

\usepackage[latin1]{inputenc}
\usepackage[T1]{fontenc}
\usepackage{array}
\usepackage[francais]{babel}
\usepackage{fontenc}
\usepackage{amsfonts,amsmath,amssymb}
\newcommand{\C}{\mathbb{C}}

\usepackage{color}

\newcommand{\D}{\textbf{D}}
\newcommand{\T}{\textbf{T}}

\newtheorem{theoreme}{Théorème}[section]
\newtheorem{proposition}[theoreme]{Proposition}
\newtheorem{remarque}[theoreme]{Remarque}

\newtheorem{lemme}[theoreme]{Lemme}

\newtheorem{definition}{Définition}
\newcommand{\N}{\mathbb{N}}
\usepackage[francais]{babel}
\begin{document}
\title{Espace de Dixmier des opérateurs de Hankel sur les espaces de Bergman à poids }
\author{Romaric Tytgat}

\address{Tytgat: LATP, U.M.R. C.N.R.S. 6632, CMI\\
         Universit\'e de Provence\\
         39 Rue F-Joliot-Curie\\
         13453 Marseille Cedex 13, France}
\email{tytgat@cmi.univ-mrs.fr}



\begin{abstract}
Nous donnons des résultats théoriques sur l'idéal de Macaev et la trace de Dixmier. Ensuite, nous caractérisons les symboles antiholomorphes $\bar{f}$ tels que l'opérateur de Hankel $H_{\bar{f}}$ sur l'espace de Bergman à poids soit dans l'idéal de Macaev et nous donnons  la trace de Dixmier. Pour cela, nous regardons le comportement des normes de Schatten $\mathcal{S}^{p}$ quand $p$ tend vers $1$ et nous nous appuyons sur le résultat de Englis et Rochberg sur l'espace de Bergman. Nous parlons aussi des puissances de tels opérateurs.
\end{abstract}

\maketitle

\section { Introduction}
Les propriétés spectrales des opérateurs de Hankel sont traités dans \cite{AFP} et \cite{A} pour les espaces de Bergman et dans l'ouvrage de \cite{Pe} pour l'espace de Hardy. On pourra aussi consulter le livre de Zhu \cite{Z} et les références qu'il contient. Les auteurs de \cite{SY} s'intéressent eux au cas de certains espaces de Fock. L'étude de leur trace de Dixmier est plus récente, même si le premier résultat dans ce sens se trouve dans \cite{AFP}, où il est montré que l'espace de Besov $\mathcal{B}^{1}$ est inclus dans l'espace de Dixmier $\mathcal{D}^{1}$. C'est une dizaine d'années plus tard que l'espace de Dixmier a été caractérisé par Li et Russo \cite{LR} dans le cas du petit opérateur de Hankel. Il a encore fallu attentre plus de dix ans pour connaître l'expression de la trace de Dixmier du grand opérateur de Hankel sur l'espace de Bergman de la boule unité de $\C^{d}, d>1$. Ce dernier résultat a été obtenu par Englis, Guo et Zhang \cite{EGZ}. Le cas $d=1$ a dû être traité de façon différente, c'est dans les travaux de Englis et Rochberg \cite{ER} que l'espace et la trace de Dixmier ont été exhibés. Les auteurs arrivent à se ramener, par une transformation unitaire, à un opérateur pseudo-différentiel, et utilisent le résultat de Wodzicki. Notre approche est  celle de \cite{LR} mais nous utilisons les résultats de \cite{ER}.\\

On note $\D$ le disque unité de $\C$ et $\mathcal{H}ol(\D)$ l'ensemble des fonctions holomorphes du disque. Pour $\alpha > -1$ un réel, on pose
\[ L^{2,\alpha}:=L^{2}((\alpha+1)(1-|z|^{2})^{\alpha}dA(z),\D)\] où $dA(z)$ est la mesure d'aire normalisée, et
\[\mathcal{A}^{2,\alpha}:=L^{2,\alpha} \cap \mathcal{H}ol(\D). \]
Soit alors $P_{\alpha}$ la projection (orthogonale) de Bergman de $L^{2,\alpha}$ sur $\mathcal{A}^{2,\alpha}$:
\[ P_{\alpha}(f)(z) = \int_{\D} f(w) K_{\alpha}(z,w) dA_{\alpha}(z) \]
où $dA_{\alpha}(z)=(\alpha+1)(1-|z|^{2})^{\alpha}dA(z)$ et $K_{\alpha}(z,w) = (1-z\bar{w})^{-(2+\alpha)}$ pour $z,w  \in \D$. \\
Pour $f \in L^{2,\alpha}$, on définit l'opérateur  $H^{\alpha}_{f}$ de $L^{2,\alpha}$ dans lui même par:
\[ H^{\alpha}_{f}(g)(z)= (I-P_{\alpha})(fP_{\alpha}g)(z). \]
Cet opérateur est densément défini et s'appelle l'opérateur de Hankel (cf \cite{AFP} et \cite{Z}). \\

Avant de continuer, faisons quelques rappels sur la métrique de Bergman, c'est l'objet du paragraphe 4.3 p 58 de \cite{Z}. On définit la distance $\beta$ sur $\D$ par:
\[ \beta(z,w) := \frac{1}{2} \log\left( \frac{1 + \rho(z,w)}{1-\rho(z,w)}\right) \quad z,w \in \D \]
où
\[\rho(z,w)=|\varphi_{z}(w)|=\left|\frac{z-w}{1-\bar{w}z}\right| \quad z,w \in \D.\]
Le disque hyperbolique de centre $z \in \D$ et de rayon $r>0$ est alors:
\[D(z,r)=\{ w \in \D ; \beta(z,w)<r \}.\]
 On a alors le lemme 4.3.6 p 62 de \cite{Z}:
\begin{lemme} \label{propriétésatomiques}
Il existe un entier $N$ tel que pour tout réel positif $r\leq1$, il existe une suite $(\lambda_{n})$ du disque vérifiant
\begin{enumerate}
	\item $\D= \cup_{n=1}^{\infty} D(\lambda_{n},r)$;
	\item $D(\lambda_{n},r) \cap D(\lambda_{m},r) = \varnothing$ si  $n \neq m$;
	\item tout point de $\D$ appartient au plus à $N$ disques $D(\lambda_{n}, 2r)$.
\end{enumerate}
\end{lemme}
Toute suite du disque vérifiant ce lemme sera appelée suite atomique.\\

Rappelons  que la p-ième classe de Schatten $\mathcal{S}^{p}$ pour $0<p\leq \infty$ est le sous espace des opérateurs compacts $T$ vérifiant $(s_{n}(T))_{n \in \N} \in l^{p}(\N)$ où $(s_{n}(T))_{n}$ est la suite des valeurs singulières rangée par ordre décroissant (cf \cite{MV} ).\\

Les auteurs de \cite{AFP} ont montré que pour $f \in \mathcal{H}ol(\D)$ et $p>1$,  $H_{\bar{f}}^{\alpha} \in \mathcal{S}^{p}$ si et seulement si $ f \in \mathcal{B}^{p}$, où
 $\mathcal{B}^{p}$ est l'espace des fonctions holomorphes du disque vérifiant:
\begin{eqnarray}\label{besovclassique}
 ||f||_{\mathcal{B}^{p}}^{p} :=\int_{\D} ((1-|z|^{2})|f'(z)|)^{p} d\lambda(z) < \infty 
\end{eqnarray}
avec $d\lambda(z)$ est la mesure Möbius invariante:
\[d\lambda(z) := \frac{1}{(1-|z|^{2})^{2}} dA(z) \]

Notons $\mathcal{S}_{1}^{+}$ l'idéal (de Macaev) des opérateurs compacts $T$ vérifiant:
\begin{displaymath}
\sup_{N \geq 2} \left \{\frac{1}{\ln{N}}\sigma_{N}(T) \right\} < \infty
\end{displaymath}
où
\begin{displaymath}
\sigma_{N}(T)=\sum_{k=0}^{N-1} s_{k}(T)
\end{displaymath}
et pour $T \in \mathcal{S}^{1}_{+}$, on note $Tr_{\omega}(T)$ sa trace de Dixmier.\\
 Moralement $Tr_{\omega}(T)= \lim_{N \rightarrow +\infty} \frac{1}{\ln{N}}\sigma_{N}(T)$. Choisissons un forme linéaire positive $\omega$ sur $l^{\infty}(\N)$ et notons sa valeur en $a=(a_{0},a_{1},...)$ par $Lim_{\omega}(a_{n})$. On demande, lorsqu'elle  existe, que  $Lim_{\omega}(a_{n})=\lim  a_{n}$ . On demande aussi que $\omega$  soit invariante par 2-dilatation, propriété technique fondamentale pour la théorie, mais qui ne nous concerne pas ici. Pour un opérateur positif $T$ dans $\mathcal{S}^{1}_{+}$ on définit sa trace de Dixmier  $Tr_{\omega}(T)$ par \[Lim_{\omega} \left(\frac{\sigma_{N}(T)}{\ln(N+1)} \right).\] $Tr_{\omega}(.)$ s'étend à $\mathcal{S}^{1}_{+}$ par linéarité. La définition de la trace de Dixmier dépend de $\omega$ et nous dirons que $T$ est mesurable si $Tr_{\omega}(T)$ est indépendant de $\omega$. Nous en verrons des exemples. On renvoie à  \cite{C} et \cite{CM}  pour plus de détails. La trace de Dixmier est une forme linéaire continue si nous équipons $\mathcal{S}^{1}_{+}$ de la norme complète:
\begin{displaymath}
||T||_{\mathcal{S}_{1}^{+}} = \sup_{N \geq 2} {\frac{1}{\ln{N}}\sum_{k=0}^{N-1} s_{k}(T) }.
\end{displaymath}

On définit alors les espaces de Dixmier $\mathcal{D}^{p}_{\alpha}$ comme ceci:

\begin{displaymath}
\mathcal{D}^{p}_{\alpha}=\{ f \in \mathcal{A}^{2,\alpha} \quad | \quad |H_{\bar{f}}^{\alpha}|^{p} \in \mathcal{S}_{1}^{+} \}
\end{displaymath}
avec
\begin{displaymath}
||f||_{\mathcal{D}^{p}_{\alpha}}=|||H^{\alpha}_{\bar{f}}|^{p}||_{\mathcal{S}_{1}^{+}}^{1/p}.
\end{displaymath}
On note simplement $\mathcal{D}^{p}$ quand $\alpha=0$.\\
De facon équivalente, on a:
\[ f \in \mathcal{D}^{p}_{\alpha} \Leftrightarrow \sum_{n=0}^{N} s_{n}(H^{\alpha}_{\bar{f}})^{p} = O(\ln(N)). \]

Si $p \geq 1$, $\mathcal{D}^{p}_{\alpha} $ est un espace vectoriel  complet pour la semie norme $||.||_{\mathcal{D}^{p}_{\alpha} }$, tandis que 
 si $0<p<1$,  $\mathcal{D}^{p}_{\alpha} $ est un espace vectoriel  complet pour la quasi-semie norme $||.||_{\mathcal{D}^{p}_{\alpha} }$ (voir \cite{T}).\\
 
 Par exemple, pour tout $k\geq0$, $H_{\bar{z}^{k}}^{\alpha} \in \mathcal{S}^{1}_{+}$ et 
\begin{eqnarray}\label{casmonome}
Tr_{\omega}(|H_{\bar{z}^{k}}^{\alpha}|)= k\sqrt{\alpha+1}.
\end{eqnarray}
Cela se calcule sans peine, puisque l'on connait les valeurs propres des opérateurs de Hankel à symbole monômial (cf \cite{AFP}).

Pour $p>0$, notons l'espace de Hardy $H^{p}$, l'ensemble des fonctions $f \in \mathcal{H}ol(\D)$ holomorphes du disque vérifiant:
\[ ||f||_{H^{p}}^{p} := \sup_{r \in [0;1[} \int_{\T} |f(re^{i\theta})| \frac{d\theta}{2\pi} < \infty. \]
 
Pour le moment, on a les résultats suivants:

\textbf{Théorème}(M. Englis et R. Rochberg \cite{ER})
Soit $f$ une fonction holomorphe sur $\D$:
\[ H_{\bar{f}} \in \mathcal{S}^{1}_{+} \Leftrightarrow  f \in H' \] 
où $H'$ est l'espace des fonctions à dérivée dans $H^{1}$.
De plus, pour $f' \in H^{1}$, on a
\[Tr_{\omega}(|H_{\bar{f}}|)=\int_{\T} |f'| d\theta\]
où $d\theta$ est la mesure de Lebesgue normalisée sur le cercle unité.\\

\textbf{Théorème}(M. Englis et R. Rochberg \cite{ER})
Soit $f \in C^{\infty}(\bar{\D})$  alors $ H_{\bar{f}} \in \mathcal{S}^{1}_{+}$ et 
\[Tr_{\omega}(|H_{f}|)=\int_{\T} |\bar{\partial}f| d\theta\]
où $d\theta$ est la mesure de Lebesgue normalisée sur le cercle unité.

Nous allons étendre ces résultat pour tout $\alpha>-1$:

\begin{theoreme}
\label{théorème1}
Soit $f$ une fonction holomorphe sur $\D$ et $\alpha>-1$:
\[ H_{\bar{f}}^{\alpha} \in \mathcal{S}^{1}_{+} \Leftrightarrow  f' \in H^{1}. \] 
De plus, pour $f' \in H^{1}$, on a
\[Tr_{\omega}(|H_{\bar{f}}^{\alpha}|)=\sqrt{\alpha+1}\int_{\T} |f'| d\theta\]
où $d\theta$ est la mesure de Lebesgue normalisée sur le cercle unité.
	\end{theoreme}
	
	\begin{theoreme}
\label{théorème2}
Soit $f \in C^{\infty}(\bar{\D})$  alors $ H^{\alpha}_{\bar{f}} \in \mathcal{S}^{1}_{+}$ et 
\[Tr_{\omega}(|H_{f}^{\alpha}|)=\sqrt{\alpha+1}\int_{\T} |\bar{\partial}f| d\theta\]
où $d\theta$ est la mesure de Lebesgue normalisée sur le cercle unité.
	\end{theoreme}
	
Un mot sur les notations, $f(z) \lesssim g(z)$ siginifie qu'il existe une constante $C>0$, pour tout  $z$ dans l'ensemble considéré, vérifiant $ f(z) \leq C g(z).$
De même pour $f(z) \gtrsim g(z)$, et on note $f(z) \simeq g(z)$  si $f(z) \gtrsim g(z)$ et $f(z) \lesssim g(z)$.
\section{Résultats généraux}

%

Pour un opérateur compact $T$, on définit , $\zeta_{T}(s):= \sum_{n=0}^{\infty} s_{n}(T)^{s}$, pour $s>1$. Si $T$ est de plus positif, alors  $\zeta_{T}(s)=Tr(T^{s})$. Notre approche est clairement motivée par le résultat de S.Y. Li et  B. Russo  \cite{LR}:	
\begin{lemme}
\label{lirusso}
Soit $T$ un opérateur compact, on a alors:
\begin{displaymath}
T \in \mathcal{S}_{1}^{+} \Leftrightarrow \sup_{s \in]1;2[}\{(s-1)\zeta_{T}(s)\} < \infty.
\end{displaymath}
\end{lemme}

On peut dire un peu plus et montrer que ces deux quantités sont équivalentes. Si $(a_{n})_{n \in \N}$ est une suite de nombres qui converge vers $0$, on note $(a^{*}_{n})_{n \in \N}$ la suite $(|a_{n}|)_{n \in \N}$ rangée par ordre décroissant. \\
Soit  $l^{1,+}$ l'ensemble des suites vérifiant $\sum_{k=0}^{n}a_{k}^{*}=O(\ln(n+1))$ muni de la norme:

\begin{displaymath}
||a||_{1,+}=\sup_{N\in\N^{*}} \left\{ \frac{1}{\ln(N+1)} \sum_{k=0}^{N}a_{k}^{*} \right\}.
\end{displaymath}

\begin{lemme}
$l^{1,+}$ est complet.
\end{lemme}
\begin{proof}
Pour tout entier $N>0$, on définit la famille de normes
\begin{displaymath}
||a||_{N}= \frac{1}{\ln(N+1)} \sum_{k=0}^{N}a_{k}^{*}. 
\end{displaymath}
En fait, voir \cite{S} p 4, on peut définir les $a_{k}^{*}$ de la façon suivante. On pose \[ a_{0}^{*}= \max_{k \in \N} |a_{k}|, a_{0}^{*} + a_{1}^{*} =\max_{k \neq j}(|a_{k}|+|a_{j}|)\] et ainsi de suite. Il est alors clair que les $||.||_{N}$ sont des normes.\\
On a:
\begin{displaymath}
\frac{1}{\ln(N+1)}||a||_{\textrm{sup}}\leq||a||_{N} \leq N||a||_{\textrm{sup}}.
\end{displaymath}
On note ici $||.||_{\sup}$ la norme infinie plutôt que $||.||_{\infty}$ afin d'éviter toute confusion avec le cas $N=\infty$.\\
Si $(a^{p})_{p \in \N}$ est une suite de Cauchy de $l^{1,+}$, elle l'est pour chaque $||.||_{N}$ et donc pour $||.||_{\sup}$. Ainsi, il existe $a\in l^{\infty}$ tel que 
\begin{displaymath}
||a^{p}-a||_{N} \rightarrow 0
\end{displaymath}
quand $p$ tend vers l'infini et pour tout $N>0$. Soit $\epsilon >0$, pour $r$ et $q$ assez grands, on a:
\begin{displaymath}
\frac{1}{\ln(N+1)}\sum_{k=0}^{N} (a^{q}_{k}-a^{r}_{k})^{*} \leq \epsilon \quad \forall N.
\end{displaymath}
A la limite sur $r$, on a:
\begin{displaymath}
\frac{1}{\ln(N+1)}\sum_{k=0}^{N} (a^{q}_{k}-a_{k})^{*} \leq \epsilon \quad \forall N
\end{displaymath}
et donc $(a^{p})$ converge vers $a$ dans $l^{1,+}$.
\end{proof}
On peut définir une autre norme sur $l^{1,+}$:
\begin{displaymath}
||a||_{\zeta}= \sup_{s\in ]1;2]} \left\{ \left((s-1)\sum |a_{k}|^{s}\right)^{1/s} \right\}.
\end{displaymath}
En effet:
\begin{lemme}
Soit $a \in l^{1,+}$. Alors
\begin{displaymath}
||a||_{\zeta} \lesssim ||a||_{1,+}.
\end{displaymath}
\end{lemme}
\begin{proof}
De l'équivalent $\ln(n)\sim \sum_{k=1}^{n} \frac{1}{k}$ en l'infini, on déduit que pour tout entier naturel $n$, on a:
\begin{displaymath}
\sum_{k=1}^{n} |a_{k}| \lesssim \sum_{k=1}^{n} \frac{1}{k} ||a||_{1,+}.
\end{displaymath}
Par le lemme 16.30 p 165 de \cite{MV}, on a pour $1<s\leq2$:
\begin{eqnarray*}
\sum_{k=1}^{n} |a_{k}|^{s} &\lesssim& \sum_{k=1}^{n} \frac{1}{k^{s}} ||a||_{1,+}^{s}\\
&\lesssim& \frac{s}{s-1} ||a||_{1,+}^{s}.\\
\end{eqnarray*}
Ainsi
\begin{displaymath}
\left((s-1) \sum_{k \in \N} |a_{k}|^{s}\right)^{1/s}  \lesssim  ||a||_{1,+}.
\end{displaymath}
Passant au sup sur les $s \in ]1;2]$, on obtient le résultat.
\end{proof}

\begin{lemme}
$l^{1,+}$ muni de $||.||_{\zeta}$ est complet.
\end{lemme}
\begin{proof}
Si $(a^{p})_{p \in \N}$ est une suite de Cauchy pour $||.||_{\zeta}$ elle l'est pour chaque $l^{s}$ et donc converge vers $a_{s} \in l^{s}$. Or si $1\leq p_{1}\leq p_{2}$, on a $||.||_{p_{1}} \geq ||.||_{p_{2}}$, ce qui montre que $a:=a_{s}$ est indépendant de $s$. Soit $\epsilon>0$, $r$ et $q$ assez grands:
\begin{eqnarray*}
||a-a^{r}||_{\zeta} &= & \sup_{s\in ]1;2]} \left\{ \left((s-1)\right)^{1/s} \limsup_{q \rightarrow + \infty} ||a^{q}-a^{r}||_{l^{s}}  \right\}\\
& \leq & \limsup_{q  \rightarrow + \infty} \sup_{s\in ]1;2]} \left\{ \left((s-1)\right)^{1/s} ||a^{q}-a^{r}||_{l^{s}}  \right\}\\
& \leq & \epsilon
\end{eqnarray*}
et donc $(a^{p})$ converge vers $a$ pour $||.||_{\zeta}$.
\end{proof}

On obtient alors 
\begin{proposition}
\label{lirusso1}
Les normes $||.||_{1,+}$ et $||.||_{\zeta}$ sont équivalentes.
\end{proposition}
\begin{proof}
Résulte des trois lemmes précédents et du théorème de Banach.
\end{proof}

Pour la trace de Dixmier, nous avons la proposition 4 p 306 de A. Connes \cite{C}:

\begin{lemme}
\label{connes}
Soit $T$ un opérateur compact positif dans $\mathcal{S}_{1}^{+}$. On a équivalence entre
\begin{displaymath}
(s-1)\zeta_{T}(s) \rightarrow L \quad \textrm{quand} \quad s \rightarrow 1^{+} 
\end{displaymath}
et 
\begin{displaymath}
(\ln{N})^{-1} \sum_{n=0}^{N-1} s_{n}(T) \rightarrow L \quad \textrm{quand} \quad N \rightarrow +\infty
\end{displaymath}
Sous ces conditions, $Tr_{\omega}(T)=L$. 
\end{lemme}

Lorsque nous n'avons pas convergence, on a le résultat classique:

\begin{lemme}  \label{encadrementtrace1}
Soit $T \in \mathcal{S}^{1}_{+}$ un opérateur positif, alors:
\begin{displaymath}
\liminf_{n  \rightarrow + \infty} \frac{1}{\ln(n)} \sigma_{n}(T) \leq Tr_{\omega}(T) \leq \limsup_{n  \rightarrow + \infty } \frac{1}{\ln(n)} \sigma_{n}(T).
\end{displaymath}
\end{lemme}

Mais on peut montrer que:
\begin{proposition}
Soit $(s_{k})$ une suite décroissante de limite nulle. On a:
\[\limsup_{s \rightarrow 1^{+}} (s-1)\sum_{k=1}^{\infty} s_{k}^{s}\leq \limsup_{N \rightarrow \infty} \frac{1}{\ln(N)}\sum_{k=1}^{N} s_{k}  \leq e \limsup_{s \rightarrow 1^{+}} (s-1)\sum_{k=1}^{\infty} s_{k}^{s} \]
\end{proposition}
\begin{proof}
On commence par montrer que
\[\limsup_{s \rightarrow 1^{+}} (s-1)\sum_{k=1}^{\infty} s_{k}^{s}\leq \limsup_{N \rightarrow \infty} \frac{1}{\ln(N)}\sum_{k=1}^{N} s_{k}. \]
Soit $\epsilon>0$, il existe $N_{0} \in \N$ vérifiant pour tout $N \geq N_{0}$
\[ \frac{1}{\ln(N)}\sum_{k=1}^{N} s_{k}  \leq \limsup_{N \rightarrow \infty} \left( \frac{1}{\ln(N)}\sum_{k=1}^{N} s_{k} \right)+ \epsilon. \]
Notons
\[ c = \limsup_{N \rightarrow \infty} \left( \frac{1}{\ln(N)}\sum_{k=1}^{N} s_{k} \right)+ \epsilon. \]

On a pour tout entier $N \geq N_{0}$
\[\frac{1}{\ln(N)}\sum_{k=N_{0}}^{N} s_{k} \leq \frac{1}{\ln(N)}\sum_{k=1}^{N} s_{k}\leq c \]
et donc
\[\frac{1}{\ln(N)}\sum_{k=1}^{N-N_{0}+1} s_{k+N_{0}-1} \leq c \]
ainsi pour tout $N-N_{0}+1 \geq 1$
\begin{eqnarray*}
\sum_{k=1}^{N-N_{0}+1} s_{k+N_{0}-1} &\leq& c\ln(N)  \\
&=& c   \ln(N)  \left(\sum_{k=1}^{N-N_{0}+1} \frac{1}{k} \right) ^{-1} \left( \sum_{k=1}^{N-N_{0}+1} \frac{1}{k} \right).\\
\end{eqnarray*}
On applique le lemme 16.30 p 165 de \cite{MV} et on obtient:
\[ \sum_{k=1}^{N-N_{0}+1} s_{n+N_{0}} ^{s} \leq \left(c   \ln(N)  \left(\sum_{k=1}^{N-N_{0}+1} \frac{1}{k} \right) ^{-1} \right)^{s} \sum_{k=1}^{N-N_{0}+1} \left(\frac{1}{k}\right)^{s}\]
et donc à la limite sur $N$
\[ \sum_{k=1}^{\infty} s_{n+N_{0}} ^{s} \leq c^{s} \sum_{k=1}^{\infty} \left(\frac{1}{k}\right)^{s}.\]
Par une comparaison série/intégrale on voit que:
\[ \sum_{k=1}^{\infty} \frac{1}{k^{s}} \leq \frac{s}{s-1} \]
ainsi
\[ (s-1)\sum_{k=1}^{\infty} s_{n+N_{0}} ^{s} \leq sc^{s}. \]
A la limite sur $s$:
\[ \limsup_{s \rightarrow 1^{+}} (s-1)\sum_{k=1}^{\infty} s_{n+N_{0}} ^{s} \leq c  \]
et comme 
\[ \limsup_{s \rightarrow 1^{+}} (s-1)\sum_{k=1}^{\infty} s_{n+N_{0}} ^{s} = \limsup_{s \rightarrow 1^{+}} (s-1)\sum_{k=1}^{\infty} s_{k} ^{s} \]
on a
\[ \limsup_{s \rightarrow 1^{+}} (s-1)\sum_{k=1}^{\infty} s_{k} ^{s} \leq c =\limsup_{N \rightarrow \infty} \left( \frac{1}{\ln(N)}\sum_{k=1}^{N} s_{k} \right)+ \epsilon . \]

Passons à présent à:
\[ \limsup_{N \rightarrow \infty} \frac{1}{\ln(N)}\sum_{k=1}^{N} s_{k}  \leq e \limsup_{s \rightarrow 1^{+}} (s-1)\sum_{k=1}^{\infty} s_{k}^{s}. \]
Soit $N$ un entier positif, notons 
\[ G_{N}=\left(\sum_{k=1}^{N} \frac{1}{k}\right)^{-1} \]
par Jensen, pour $s>1$, on a:
\[ \left(\sum_{k=1}^{N} \frac{s_{k}}{k} G_{N} \right)^{s} \leq  \sum_{k=1}^{N} \frac{1}{k} G_{N} s_{k}^{s} \]
et en remplacant $s_{k}$ par $ks_{k}$ on obtient
\[ \left(\sum_{k=1}^{N} s_{k}  G_{N} \right)^{s} \leq  \sum_{k=1}^{N} k^{s-1} G_{N} s_{k}^{s} \leq N^{s-1}G_{N}\sum_{k=1}^{N} s_{k}^{s}\]
d'ou

\[\left( \frac{1}{\ln(N)}\sum_{k=1}^{N} s_{k}  \right)^{s} \leq \frac{N^{s-1}G_{N}^{1-s}}{\ln(N)^{s}(s-1)} (s-1)\sum_{k=1}^{N} s_{k}^{s}. \]
Posons\[ s= 1+\frac{1}{\ln(N)}\] on a
\[ \frac{N^{s-1}G_{N}^{1-s}}{\ln(N)^{s}(s-1)} = \frac{G_{N}^{-s}}{\ln(N)^{s}} \frac{N^{s-1}G_{N}}{(s-1)}. \]
Mais
\[\frac{G_{N}^{-s}}{\ln(N)^{s}}=\left(\frac{\ln(N)+\gamma+o(1)}{\ln(N)}\right)^{1+\frac{1}{\ln(N)}} \rightarrow 1 \]
et
\[ \frac{N^{s-1}G_{N}}{(s-1)} =\frac{e^{(s-1)\ln(N)}}{(s-1)(\ln(N)+\gamma+o(1))} \rightarrow e. \]
Ainsi pour $s= 1+\frac{1}{\ln(N)}$
\begin{eqnarray*}
\limsup_{N \rightarrow \infty } \frac{1}{\ln(N)}\sum_{k=1}^{N} s_{k} &\leq& e \limsup_{N \rightarrow \infty} (s-1)\sum_{k=1}^{N} s_{k}^{s} \\
&\leq& e \limsup_{t \rightarrow 1^{+}} (t-1)\sum_{k=1}^{N} s_{k}^{t} .
\end{eqnarray*}
\end{proof}
De même:
\begin{proposition}
Soit $(s_{k})$ une suite décroissante de limite nulle. On a:
\[\liminf_{s \rightarrow 1^{+}} (s-1)\sum_{k=1}^{\infty} s_{k}^{s}\geq \liminf_{N \rightarrow \infty} \frac{1}{\ln(N)}\sum_{k=1}^{N} s_{k}  \geq e^{-1} \liminf_{s \rightarrow 1^{+}} (s-1)\sum_{k=1}^{\infty} s_{k}^{s}. \]
\end{proposition}
\begin{proof}
De la même façon on montre que
\[\liminf_{s \rightarrow 1^{+}} (s-1)\sum_{k=1}^{\infty} s_{k}^{s}\geq \liminf_{N \rightarrow \infty} \frac{1}{\ln(N)}\sum_{k=1}^{N} s_{k}. \]
Passons  à:
\[ \liminf_{N \rightarrow \infty} \frac{1}{\ln(N)}\sum_{k=1}^{N} s_{k}  \geq e^{-1} \liminf_{s \rightarrow 1^{+}} (s-1)\sum_{k=1}^{\infty} s_{k}^{s}. \]
Sans perte de généralité on suppose $s_{k} \leq 1$.
Soit $N$ un entier positif, notons 
\[ G_{N}=\left(\sum_{k=1}^{N} \frac{1}{k}\right)^{-1} \]
par Jensen, pour $s<1$,on a:
\[ \left(\sum_{k=1}^{N} \frac{s_{k}}{k} G_{N} \right)^{s} \geq  \sum_{k=1}^{N} \frac{1}{k} G_{N} s_{k}^{s} \]
et en remplacant $s_{k}$ par $ks_{k}$ et comme $s_{k}<1$, on obtient
\[ \left(\sum_{k=1}^{N} s_{k}  G_{N} \right)^{s} \geq  \sum_{k=1}^{N} k^{s-1} G_{N} s_{k}^{s} \geq \sum_{k=1}^{N} k^{s-1} G_{N} s_{k}^{t} \]
où $t>1$.
Mais comme $s<1$
\[ \sum_{k=1}^{N} k^{s-1} G_{N} s_{k}^{t} \geq N^{s-1}G_{N}\sum_{k=1}^{N} s_{k}^{t}.
\]

On a
\[\left( \frac{1}{\ln(N)}\sum_{k=1}^{N} s_{k}  \right)^{s} \geq \frac{N^{s-1}G_{N}^{1-s}}{\ln(N)^{s}(t-1)}\left[ (t-1)\sum_{k=1}^{N} s_{k}^{t}\right]. 
\]
Posons\[ s= 1-\frac{1}{\ln(N)}\] et \[t= 1+\frac{1}{\ln(N)}\] on a
\[ \frac{N^{s-1}G_{N}^{1-s}}{\ln(N)^{s}(t-1)} = \frac{G_{N}^{-s}}{\ln(N)^{s}} \frac{N^{s-1}G_{N}}{(t-1)}. \]
Mais
\[\frac{G_{N}^{-s}}{\ln(N)^{s}}=\left(\frac{\ln(N)+\gamma+o(1)}{\ln(N)}\right)^{1-\frac{1}{\ln(N)}} \rightarrow 1 \]
et
\[ \frac{N^{s-1}G_{N}}{(t-1)} =\frac{e^{-1}}{(t-1)(\ln(N)+\gamma+o(1))} \rightarrow e^{-1}. \]
Ainsi 
\begin{eqnarray*}
\liminf_{N \rightarrow \infty } \frac{1}{\ln(N)}\sum_{k=1}^{N} s_{k} &\geq& e^{-1} \liminf_{N \rightarrow \infty} (t-1)\sum_{k=1}^{N} s_{k}^{t} \\
&\geq& e^{-1} \liminf_{r \rightarrow 1^{+}} (r-1)\sum_{k=1}^{N} s_{k}^{r}.
\end{eqnarray*}

\end{proof}

On obtient immédiatement:
\begin{proposition}
\label{connes1}
Soit $T \in \mathcal{S}^{1}_{+}$, alors:
\begin{displaymath}
e^{-1} \liminf_{s  \rightarrow 1}  (s-1) ||T||_{\mathcal{S}^{s}}^{s} \leq Tr_{\omega}(T) \leq  e \limsup_{s  \rightarrow 1} (s-1) ||T||_{\mathcal{S}^{s}}^{s}
\end{displaymath}
où de façon équivalente, pour reprendre la notation de Connes, Li et Russo
\begin{displaymath}
e^{-1} \liminf_{s  \rightarrow 1} (s-1) \zeta_{T}(s) \leq Tr_{\omega}(T) \leq e \limsup_{s  \rightarrow 1} (s-1) \zeta_{T}(s).
\end{displaymath}
\end{proposition}

\section{L'espace H'}
La section précédente nous invite à regarder la limite de $(p-1)||H_{\bar{f}}||_{\mathcal{S}^{p}}$.  Comme les normes $||H_{\bar{f}}||_{\mathcal{S}^{p}}$ et $||f||_{\mathcal{B}^{p}}$ sont équivalentes, nous nous intéressons au sup de $(p-1)||f||_{\mathcal{B}^{p}}$. Tous les résultats sur les espaces de Hardy utilisés dans cette section se trouvent dans \cite{RuARC}.

\begin{definition}
$H'$ est défini comme suit:
\begin{displaymath}
H'=\{h ; h' \in H^{1}\}. 
\end{displaymath}
\end{definition}

\textbf{Proposition}(J. Arazy, S. Fisher et J. Peetre \cite{AFP85})

L'espace $H'$ muni de la semi norme
\begin{displaymath}
||h||_{H'}=\int_{\T} |h'(e^{i\theta})| d\theta 
\end{displaymath}
est Möbius invariant au sens de \cite{AFP85}, et donc complet.\\
\newline
On s'inspire de \cite{AFP85} pour donner une description alternative:
\begin{proposition}
\label{descriptionhardy}
\begin{displaymath}
H'=\{h \in \mathcal{H}ol(\D); \sup_{p \in]1;2]} (p-1)||h||_{\mathcal{B}^{p}} < \infty\}
\end{displaymath}
et 
\begin{eqnarray*}
||h||_{H'} &=&||h'||_{H^{1}}\\
 &=& \lim_{p\rightarrow 1^{+}} (p-1)||h||_{\mathcal{B}^{p}}.\\
\end{eqnarray*}
\end{proposition}

\begin{proof}
Soit \[u_{p}=\int_{0}^{2\pi} |h'(re^{i\theta})|^{p} d\theta\] et \[d\lambda_{p}(r)=2(p-1)r(1-r^{2})^{p-2}dr\] suite de mesure de probabilité sur $[0;1[$. Avec ces notations, on a donc \[(p-1)||h||_{\mathcal{B}^{p}}^{p} = \int_{0}^{1} u_{p}d\lambda_{p}(r).\] On procède par double inégalité, soit $a<1$

\begin{eqnarray*}
\int_{0}^{a}u_{p} d\lambda_{p}&=& \int_{0}^{a} u_{p} \frac{d}{dr}\left(-(1-r^{2})^{p-1}\right)dr \\
&=& \int_{0}^{a} \left(\frac{d}{dr}(u_{p})\right) \left((1-r^{2})^{p-1}\right)dr + \left[ -u_{p}(1-r^{2})^{p-1}\right]_{0}^{a}\\
&\leq& \int_{0}^{a} \frac{d}{dr}(u_{p}) dr  -u_{p}(a)(1-a^{2})^{p-1}+u_{p}(0)\\
&\leq& u_{p}(a)-u_{p}(0) -u_{p}(a)(1-a^{2})^{p-1}+u_{p}(0)\\
&=& \left[1-(1-a^{2})^{p-1}\right]u_{p}(a).\\
\end{eqnarray*}
Ainsi, si $h'\in H^{1}$, $h'$ a une limite radiale pp que l'on note $h'_{*}$
\begin{eqnarray*}
\int_{0}^{1}u_{p} d\lambda_{p}&\leq & \limsup_{a\rightarrow1}\left[1-(1-a^{2})^{p-1}\right]u_{p}(a)\\
&\leq& \limsup_{a\rightarrow1} \int_{0}^{2\pi} |h'(ae^{i\theta})|^{p} d\theta\\
&\leq& \int_{0}^{2\pi} \limsup_{a\rightarrow1} |h'(ae^{i\theta})|^{p} d\theta\\
&=& \int_{0}^{2\pi} |h_{*}'(e^{i\theta})|^{p} d\theta\\
&=& ||h'_{*}||_{H^{p}}^{p}.\\
\end{eqnarray*}
On a donc
\begin{eqnarray*}
\limsup_{p\rightarrow 1^{+}} \int_{0}^{1}u_{p} d\lambda_{p}&\leq & \limsup_{p\rightarrow 1^{+}} \int_{0}^{2\pi} |h_{*}'(e^{i\theta})|^{p} d\theta\\
&\leq& \int_{0}^{2\pi}\limsup_{p\rightarrow 1^{+}}  |h_{*}'(e^{i\theta})|^{p} d\theta \\ 
&=& \int_{0}^{2\pi} |h'_{*}(e^{i\theta})| d\theta\\ 
&=& ||h'_{*}||_{H^{1}}.\\
\end{eqnarray*}
Ainsi
\[ h'\in H^{1} \Rightarrow \sup_{p\in]1;2]} \int_{0}^{1}u_{p} d\lambda_{p}(r)< \infty. \]

Réciproquement, supposons que \[ \sup_{p\in]1;2]} \int_{0}^{1}u_{p} d\lambda_{p}(r) < \infty.\]	 Par Hölder, on a:
\begin{eqnarray*}
\left(\int_{\D} |h'| d\theta d\lambda_{p}(r)\right)^{p} &\leq& \left(\int_{\D} |h'|^{p} d\theta d\lambda_{p}(r) \right) \left(\int_{\D}  d\theta d\lambda_{p}(r) \right)^{p/q} \\
&=&\left(\int_{\D} |h'|^{p} d\theta d\lambda_{p}(r) \right) 
\end{eqnarray*}
et donc
\begin{eqnarray*}
\liminf_{p\rightarrow 1^{+}} \left(\int_{\D} |h'| d\theta d\lambda_{p}(r)\right)^{p} &\leq& \liminf_{p\rightarrow 1^{+}} \left(\int_{\D} |h'|^{p} d\theta d\lambda_{p}(r) \right).
\end{eqnarray*}
Mais par \cite{AFP85}
\begin{eqnarray*}
\lim_{p\rightarrow 1^{+}} \left(\int_{\D} |h'| d\theta d\lambda_{p}(r)\right)^{p} &=& ||h'||_{H^{1}}
\end{eqnarray*}
ce qui donne
\begin{eqnarray*}
||h'||_{H^{1}} &\leq& \liminf_{p\rightarrow 1^{+}} \int_{0}^{1}u_{p} d\lambda_{p}(r)\\
&\leq& \sup_{p \in ]1;2]} \int_{0}^{1}u_{p} d\lambda_{p}(r).\\
\end{eqnarray*}

On a donc montré que \[h' \in H^{1} \Leftrightarrow \sup_{p\in]1;2]} \int_{0}^{1}u_{p} d\lambda_{p}(r) < \infty\] et que dans ce cas

\begin{eqnarray*}
||h'||_{H^{1}} &\leq& \liminf_{p\rightarrow 1^{+}} \int_{0}^{1}u_{p} d\lambda_{p}(r)\\
&\leq& \limsup_{p\rightarrow 1^{+}} \int_{0}^{1}u_{p} d\lambda_{p}(r) \\
&\leq& ||h||_{H'}\\
\end{eqnarray*}
c'est à dire
\begin{eqnarray*}
||h||_{H'}&=& \lim_{p\rightarrow 1^{+}} \int_{0}^{1}u_{p} d\lambda{p}(r).\\
\end{eqnarray*}
\end{proof}

\begin{definition}

On introduit sur $H'$ la norme

\begin{displaymath}
||f||_{sup}=\sup_{p \in]1;2[}  (p-1)^{1/p} ||f||_{\mathcal{B}^{p}}. 
\end{displaymath}
\end{definition}

\begin{lemme}
L'espace $H'$ muni de la norme $||.||_{sup}$ est complet.
\end{lemme}
\begin{proof}
Soit $(f_{n})$ une suite de Cauchy de $H'$, elle est de Cauchy dans chaque $\mathcal{B}^{p}$, $1<p<2$, donc converge dans $\mathcal{B}^{p}$ vers $f^{p}$. De l'inégalité 
\begin{displaymath}
||f||_{\mathcal{D}} \leq c_{p} ||f||_{\mathcal{B}^{p}}
\end{displaymath}
on déduit que les $f^{p}$ sont égales, on note cette limite $f$. On a noté ici $\mathcal{D}$ l'espace de Dirichlet qui est en fait l'espace de Besov $\mathcal{B}^{2}$. 
Enfin, soit $\epsilon>0$, pour $n,k$ assez grand on a:
\begin{eqnarray*}
\epsilon &>& \limsup_{k} \sup_{p} (p-1)^{1/p}||f_{n}-f_{k}||_{\mathcal{B}^{p}}\\
&\geq& \sup_{p} (p-1)^{1/p} \limsup_{k} ||f_{n}-f_{k}||_{\mathcal{B}^{p}} \\
&=& ||f- f_{n}||_{sup}.
\end{eqnarray*}
\end{proof}

\begin{proposition}
\label{equivalencehardy}
Pour tout $f \in H'$, on a
\begin{displaymath}
\sup_{p \in]1;2[}  (p-1)^{1/p} ||f||_{\mathcal{B}^{p}} \simeq \lim_{p \rightarrow 1 }  (p-1)^{1/p} ||f||_{\mathcal{B}^{p}} 
\end{displaymath}

\end{proposition}
\begin{proof}
Résulte du théorème de Banach.
\end{proof}

\section{Preuves}

On montre dans cette section les théorèmes. Pour cela, on prouve que pour tout $\alpha>-1$ , $ \mathcal{D}^{1}_{\alpha} \subset H'$ et on met en lien les différents opérateurs de Hankel pour conclure.\\
\newline
On donne une nouvelle preuve de l'inclusion $\mathcal{D}^{1} \subset H'$:
\begin{proposition}  On a:
\[\mathcal{D}^{1} \subset H'. \]
\end{proposition}

\begin{proof}
Nous montrons que nous avons, uniformément en $p$:
\[||f||_{\mathcal{B}^{p}} \lesssim ||H_{\bar{f}}||_{\mathcal{S}^{p}}\] ce qui permettra d'affirmer que $\mathcal{D}^{1}\subset H'$. Soit $f$ une fonction holomorphe dans l'espace de Dixmier, $f$ est donc dans chaque classe de Schatten $p>1$ et donc $f$ est dans l'intersection de tous les espaces de Besov d'exposant $p>1$. On note $D(z)$ le disque hyperbolique de centre $z$ et de rayon $r$ fixé. On part de l'inégalité \cite{ZABS} p 330 par exemple, pour $\lambda \in \C$
\begin{eqnarray*}
(1-|z|^{2})|f'(z)| &\lesssim& \int_{D(0)} |f \circ \varphi_{z}(w)-\lambda| dA(w) \\
&\lesssim& \left( \int_{D(0)} |f \circ \varphi_{z}(w)-\lambda|^{2} dA(w) \right)^{1/2} \\
&\lesssim& \left( \int_{D(z)} |f(w)-\lambda|^{2} |\varphi'_{z}(w)|^{2}dA(w) \right)^{1/2} \\
&\simeq& \left( \frac{1}{|D(z)|} \int_{D(z)} |f(w)-\lambda|^{2} dA(w) \right)^{1/2} \\
&\simeq& \left( \frac{1}{|D(z)|} \int_{D(z)} |\bar{f}(w)-\bar{\lambda}|^{2} dA(w) \right)^{1/2}. \\
\end{eqnarray*}
Passant à l'inf sur tous les $\lambda$ et utilisant la proposition 1 p 261 de \cite{L}, on obtient:
\begin{eqnarray} \nonumber
(1-|z|^{2})|f'(z)| &\lesssim& \inf_{\lambda \in \C} \left( \frac{1}{|D(z)|} \int_{D(z)} |\bar{f}(w)-\bar{\lambda}|^{2} dA(w) \right)^{1/2} \\ \nonumber
&\lesssim& \inf_{h \in \mathcal{A}^{2}} \left( \frac{1}{|D(z)|} \int_{D(z)} |\bar{f}(w)-h(w)|^{2} dA(w) \right)^{1/2} \\ \nonumber
&=& F(z)
\end{eqnarray}
où $F$ est la fonction de \cite{L} définie p 252:
\[ F(z)^{2}=\inf \left\{ \frac{1}{D(z)}\int_{D(z)} |\bar{f}-h|^{2} dA ; h \in \mathcal{A}^{2} \right\}. \]
 Mais en regardant attentivement la preuve du théorème 4 p 262 \cite{L}, on voit qu'indépendamment en $p$, on a:
\[ ||F||_{L^{p}(d\lambda)} \lesssim ||H_{\bar{f}}||_{\mathcal{S}^{p}}. \]
Ainsi 
\begin{eqnarray} \label{inégalitécaszero}
||f||_{\mathcal{B}^{p}} =||(1-|z|^{2})|f'(z)|||_{L^{p}(d\lambda)} \lesssim  ||F||_{L^{p}(d\lambda)} \lesssim ||H_{\bar{f}}||_{\mathcal{S}^{p}}
\end{eqnarray}
d'où
\[ \sup_{p \in ]1;2]} (p-1) ||f||_{\mathcal{B}^{p}}  \lesssim \sup_{p \in ]1;2]} (p-1)||H_{\bar{f}}||_{\mathcal{S}^{p}} \]
ce qui montre l'inclusion $\mathcal{D}^{1} \subset H'$, via la proposition \ref{descriptionhardy}.\\
\end{proof}
Passsons au cas général:
\begin{lemme}
Pour tout $\alpha>-1$
\[\mathcal{D}^{1}_{\alpha} \subset H'. \] 
\end{lemme}

\begin{proof}
Dans \cite{L}, il est dit p 267 que le théorème 4 reste inchangé modulo les adaptations dûes au poids. On précise les quelques modifications afin d'être complet.
Reprenons la preuve de \cite{L} et remarquons que l'opérateur $A$ est l'opérateur $-T$ de \cite{Z} p 66 de $l^{2}$ dans $\mathcal{A}^{2}$:
\[ A((a_{k}))(z)=\sum_{k=1}^{\infty} a_{k} \frac{(1-|\xi_{k}|^{2})}{(1-\bar{\xi}_{k}z)^{2}} \]  où $(\xi_{k})$ est une suite atomique. Cela résulte simplement de:
\[ -\varphi'_{\xi_{k}}(z)= \frac{(1-|\xi_{k}|^{2})}{(1-\bar{\xi}_{k}z)^{2}}.\]
Nous définissons donc $A_{\alpha}$ de $l^{2}$ dans $\mathcal{A}^{2,\alpha}$ comme suit:
\[ A_{\alpha} ( (a_{k})) = \sum_{k=1}^{\infty} a_{k} \frac{(1-|\xi_{k}|^{2})^{(2+\alpha)/2}}{(1-\bar{\xi}_{k}z)^{2+\alpha}}\]
et l'on montre qu'il est borné de la même façon que dans \cite{Z} p 66.  Soit donc $f \in \mathcal{H}^{\infty}$ une fonction holomorphe bornée du disque. On a:
\[ \langle A_{\alpha}( (a_{k})) , f \rangle_{\mathcal{A}^{2,\alpha}} = \sum_{k=1}^{\infty}  a_{k} (1-|\xi_{k}|^{2})^{(2+\alpha)/2}  \bar{f}(\xi_{k})\] 
Ainsi par Cauchy-Schwarz
\[ |\langle A_{\alpha}( (a_{k})) , f \rangle_{\mathcal{A}^{2,\alpha}}| \leq ||(a_{k})||_{l^{2}} ||(1-|\xi_{k}|^{2})^{(2+\alpha)/2}  \bar{f}(\xi_{k})||_{l^{2}} \]
Comme ${H}^{\infty}$ est dense dans  $\mathcal{A}^{2,\alpha}$, il suffit de montrer que 
\[ ||(1-|\xi_{k}|^{2})^{(2+\alpha)/2}  \bar{f}(\xi_{k})||_{l^{2}} \lesssim ||f||_{\mathcal{A}^{2,\alpha}} \quad, \quad f \in {H}^{\infty} \]
pour conclure.\\
Soit $z \in \D$ et $r>0$, par le lemme 4.3.4 p 61 de \cite{Z}:
\[ \frac{1}{|D(z,r)|} \int_{D(z,r)} (1-|w|^{2})^{\alpha}|f(w)|^{2} dA(w) \simeq \frac{(1-|z|^{2})^{\alpha}}{|D(z,r)|} \int_{D(z,r)} |f(w)|^{2} dA(w) \]
et par la proposition 4.3.8 p 62 de \cite{Z} on déduit que:
\[ (1-|z|^{2})^{\alpha}|f(z)|^{2} \lesssim \frac{1}{|D(z,r)|} \int_{D(z,r)} (1-|w|^{2})^{\alpha}|f(w)|^{2} dA(w).\]
Comme $|D(z,r)|$ est comparable à $(1-|z|^{2})^{2}$ (lemme 4.3.3 p 60 \cite{Z}) on a:
\[ (1-|z|^{2})^{\alpha+2}|f(z)|^{2} \lesssim \int_{D(z,r)} (1-|w|^{2})^{\alpha}|f(w)|^{2} dA(w).\]
Pour $z=\xi_{k}$ et à l'aide du (3) du lemme \ref{propriétésatomiques}, on obtient
\[ \sum_{k=1}^{\infty} (1-|\xi_{k}|^{2})^{\alpha+2}|f(\xi_{k})|^{2} \lesssim \int_{\D} (1-|w|^{2})^{\alpha}|f(w)|^{2} dA(w)\]
ce qui prouve bien:
\[ ||(1-|\xi_{k}|^{2})^{(2+\alpha)/2}  \bar{f}(\xi_{k})||_{l^{2}} \lesssim ||f||_{\mathcal{A}^{2,\alpha}} \quad, \quad f \in {H}^{\infty} \]
et donc $A_{\alpha}$ est borné de $l^{2}$ dans $\mathcal{A}^{2,\alpha}$.\\
On définit aussi l'opérateur $B_{\alpha}$ de $l^{2}$ dans $\mathcal{A}^{2,\alpha}$ qui envoie le $k$-ième vecteur $e_{k}$ de la base canonique de $l^{2}$ sur $c_{k} \chi_{D(\xi_{k})}H_{\bar{f}}^{\alpha}((\varphi'_{k})^{1+\alpha/2})$ où $c_{k}$ est une constante pour que $B_{\alpha}e_{k}$ soit unitaire. Cet opérateur est borné puisque  $(\xi_{k})$ est une suite atomique. Nous avons alors comme dans \cite{L}:
\begin{eqnarray*}
\sum_{k} \left(\int_{D(0)} |H_{\bar{f} \circ \varphi_{k}}^{\alpha}(1)|^{2} dA \right)^{p/2}&=& \sum_{k} \left(\int_{D(0)} |(\varphi'_{k} H_{\bar{f}}^{\alpha}((\varphi'_{k})^{1+\alpha/2}))(\varphi_{k})|^{2} dA\right)^{p/2}\\
&=& \sum_{k} \left(\int_{D(\xi_{k})} | H_{\bar{f}}^{\alpha}((\varphi'_{k})^{1+\alpha/2})|^{2} dA\right)^{p/2}\\
&=&\sum_{k}|\langle B^{*} H^{\alpha}_{\bar{f}} T_{\alpha} e_{k}, e_{k} \rangle|^{p}.\\
&\lesssim& ||H^{\alpha}_{\bar{f}}||_{\mathcal{S}^{p}}^{p}\\
\end{eqnarray*}
où la dernière ligne est justifiée par \cite{GK} p 95.\\
On a aussi, comme au début de la page 265 de \cite{L}:
\[ F(\xi_{k})^{2} \lesssim \int_{D(\xi_{k})} | H_{\bar{f}}^{\alpha}((\varphi'_{k})^{1+\alpha/2})|^{2} dA = \int_{D(0)} |H_{\bar{f} \circ \varphi_{k}}^{\alpha}(1)|^{2} dA \] 
où l'on rappelle que
\[ F(z)^{2}=\inf \left\{ \frac{1}{D(z)}\int_{D(z)} |\bar{f}-h|^{2} dA ; h \in \mathcal{A}^{2} \right\}.\]
Enfin, au bas de la page 266 de \cite{L}:
\[ ||F||_{L^{p}(d\lambda)} \lesssim ||F(\xi_{k})||_{l^{p}}\]
indépendamment en $p>1$. On obtient alors comme dans le cas $\alpha=0$, indépendamment en $p>1$:
\begin{eqnarray}
\label{inclusionpoids}
 ||F||_{L^{p}(d\lambda)} \lesssim || H_{\bar{f}}^{\alpha}||_{\mathcal{S}^{p}}.  
 \end{eqnarray}
Or, on a vu (\ref{inégalitécaszero}) que:
\[ ||f||_{\mathcal{B}^{p}} \simeq  ||F||_{L^{p}(d\lambda)}. \] On conclut alors comme dans le théorème précédent.
\end{proof}

On a donc montré que pour tout $\alpha>-1$, $\mathcal{D}^{1}_{\alpha} \subset H'$. Or, $H' \subset H^{\infty}$ (p 79 \cite{Pa}), l'espace des fonctions holomorphes bornées du disque. Ainsi toute fonction dans $\mathcal{D}^{1}_{\alpha}$ est bornée sur le disque.

\begin{lemme}
Pour tout $\alpha>-1$ et pour tout symbole borné $f$ 
\[H_{\bar{f}} \in \mathcal{S}^{1}_{+} \Leftrightarrow H_{\bar{f}}^{\alpha} \in \mathcal{S}^{1}_{+}.\]
\end{lemme}

\begin{proof}
On essaye de trouver un lien entre la projection de Bergman classique et celles à poids.
 On note $C_{\alpha}$ l'isométrie de $L^{2,\alpha}$ dans $L^{2}$ définie par
\[ C_{\alpha}(f)=(1-|z|)^{\alpha/2}f \]
et $P_{\alpha}$ la projection de Bergman. On regarde l'image d'un monôme par $P_{0}C_{\alpha}$:
\begin{eqnarray*}
P_{0}C_{\alpha}(w^{n})(z)&=&P_{0}((1-|w|^{2})^{\alpha/2}w^{n})(z)\\
&=&\int_{\D}\frac{w^{n}(1-|w|^{2})^{\alpha/2}}{(1-\bar{w}z)^{2}}dA(w)\\
&=&\int_{0}^{1} r^{n}(1-r^{2})^{\alpha/2} \int_{\T}\frac{\xi^{n}}{(1-r\bar{\xi}z)^{2}}d\xi rdr\\
&=&\int_{0}^{1} r^{n}(1-r^{2})^{\alpha/2} \int_{\T}\xi^{n}\sum_{k} (k+1)r^{k}\bar{\xi}^{k} z^{k} d\xi rdr\\
&=&\int_{0}^{1} r^{2n}(1-r^{2})^{\alpha/2}  (n+1) z^{n}  rdr\\
&=&\int_{0}^{1} r^{n}(1-r)^{\alpha/2}  (n+1) z^{n} dr\\
&=& (n+1)\beta(n+1,\alpha/2+1)z^{n} \\
\end{eqnarray*}
où $\beta$ est la fonction Bêta. Ainsi, si l'on définit l'opérateur $M_{\alpha}$ de $\mathcal{A}^{2}$ dans $\mathcal{A}^{2,\alpha}$ par 
\[ M_{\alpha}(z^{n})= [(n+1)\beta(n+1,\alpha/2+1)]^{-1}z^{n}\]
on a l'identité sur $\mathcal{A}^{2,\alpha}$:
\[M_{\alpha}P_{0}C_{\alpha}=Id=P_{\alpha}.\]
Cet opérateur $M_{\alpha}$ est borné. En effet:
 \[(\sqrt{n+1}z^{n})_{n} \quad \textrm{et} \quad ([\sqrt{\alpha+1}\sqrt{\beta(n+1,\alpha+1)}]^{-1}z^{n})_{n}\] sont des bases orthonormales respectivement de $\mathcal{A}^{2}$ et $\mathcal{A}^{2,\alpha}$ \cite{AFP}. Ainsi
\begin{eqnarray*}
M_{\alpha}(\sqrt{n+1}z^{n})&=& \frac{1}{\sqrt{n+1}}\frac{1}{\beta(n+1,\alpha/2+1)}z^{n}\\
&=&\frac{\sqrt{\alpha+1}\sqrt{\beta(n+1,\alpha+1)}}{\sqrt{n+1}\beta(n+1,\alpha/2+1)} \frac{1}{\sqrt{\alpha+1}\sqrt{\beta(n+1,\alpha+1)}} z^{n}. 
\end{eqnarray*}
De l'équivalent en l'infini, pour $y$ positif fixé \[ \beta(x,y) \sim \Gamma(y) x^{-y} \] on en déduit que 
\begin{eqnarray*}
\frac{\sqrt{\alpha+1}\sqrt{\beta(n+1,\alpha+1)}}{\sqrt{n+1}\beta(n+1,\alpha/2+1)} &\simeq& \frac{(n+1)^{\alpha/2+1}}{(n+1)^{1/2}(n+1)^{\alpha/2+1/2}}\\
&\simeq& 1\\
\end{eqnarray*}
ce qui suffit pour conclure. L'opérateur $M_{\alpha}$ est même inversible et on le prolonge de façon artificielle sur $L^{2}$ tout entier en envoyant une base orthonormale de $\left(\mathcal{A}^{2}\right)^{\perp}$ sur une base orthonormale de $\left(\mathcal{A}^{2,\alpha}\right)^{\perp}$. Cela reste un opérateur borné et inversible.

%
On peut aussi montrer que notre opérateur $M_{\alpha}P_{0}C_{\alpha}$ coïncide avec $P_{\alpha}$ sur $\bar{\mathcal{A}}^{2,\alpha}$ c'est à dire que 
 $M_{\alpha}P_{0}C_{\alpha}(\bar{\mathcal{A}}^{2,\alpha})=0$. Cependant $M_{\alpha}P_{0}C_{\alpha}$ n'envoie a priori pas  l'orthogonal de $\mathcal{A}^{2,\alpha}$ sur $0$. Il faut donc introduire l'opérateur $X$ sur $L^{2,\alpha}$ tel que
\[M_{\alpha}P_{0}C_{\alpha}+X = P_{\alpha}.\]

  Sur $L^{2, \alpha}$, on a
\begin{eqnarray*}
I-P_{\alpha}&=&I-M_{\alpha}P_{0}C_{\alpha}-X\\
&=& M_{\alpha}(I-P_{0})C_{\alpha} + I - M_{\alpha}C_{\alpha}-X\\
\end{eqnarray*}
et donc pour $f\in L^{\infty}$, où l'on note $M_{\bar{f}}$ l'opérateur de multiplication par $\bar{f}$ sur $L^{2,\alpha}$:
\begin{eqnarray*}
(I-P_{\alpha})M_{\bar{f}}= M_{\alpha}(I-P_{0})C_{\alpha}M_{\bar{f}} + (I - M_{\alpha}C_{\alpha}-X)M_{\bar{f}}.
\end{eqnarray*}
En voyant  $M_{\bar{f}}$ comme un opérateur sur $L^{2}$, on peut écrire
\begin{eqnarray*}
(I-P_{\alpha})M_{\bar{f}}= M_{\alpha}(I-P_{0})M_{\bar{f}}C_{\alpha} + (I - M_{\alpha}C_{\alpha}-X)M_{\bar{f}}
\end{eqnarray*}
d'où
\begin{eqnarray}\label{eq1}
H_{\bar{f}}^{\alpha}=   M_{\alpha}H_{\bar{f}}C_{\alpha} + (I - M_{\alpha}C_{\alpha}-X)M_{\bar{f}}. 
\end{eqnarray}
Posons $f(z)=z+2$, par l'égalité (\ref{casmonome}), on sait que $H_{\bar{f}}^{\alpha}$ et $H_{\bar{f}}$ sont dans l'idéal de Macaev et donc $(I - M_{\alpha}C_{\alpha}-X)M_{\bar{f}}$ l'est aussi. Comme $M_{\bar{f}}$ est inversible et borné, on en déduit que $(I - M_{\alpha}C_{\alpha}-X)$ est aussi dans l'idéal de Macaev. Or cet opérateur est indépendant de $\bar{f}$ donc si l'on revient à l'égalité pour $\bar{f}$ borné quelconque:
\[H_{\bar{f}}^{\alpha}=M_{\alpha}H_{\bar{f}}C_{\alpha} + (I - M_{\alpha}C_{\alpha}-X)M_{\bar{f}} \]
on voit que pour tout symbole borné $f$, puisque $M_{\alpha}$ est inversible borné, $H_{\bar{f}}^{\alpha}$ est dans l'idéal de Macaev si et seulement si $H_{\bar{f}}$ l'est .
\end{proof}
\begin{remarque}
On aurait pu construire, au moins de façon théorique, une isométrie $C_{\alpha}$ envoyant $\mathcal{A}^{2,\alpha}$ sur $\mathcal{A}^{2}$ et de même pour leurs orthogonaux. Cependant, on n'aurait pu, à priori, le faire commuter avec $M_{\bar{f}}$.
\end{remarque}
Comme $\mathcal{D}^{1}_{\alpha} \subset H^{\infty}$, cela suffit pour affirmer que l'espace de Dixmier est l'espace $H'$ pour tout $\alpha>-1$. \\

Passons à l'expression de la trace, pour cela, remarquons que l'égalité (\ref{eq1}):
\begin{eqnarray*}
H_{\bar{f}}^{\alpha}=   M_{\alpha}H_{\bar{f}}C_{\alpha} + (I - M_{\alpha}C_{\alpha}-X)M_{\bar{f}}. 
\end{eqnarray*}
 est vraie pour une fonction $f$ positive dans $C_{c}(\D)$ et $g$ positive dans $C_{0}(\D)$ et notons \[T=I - M_{\alpha}C_{\alpha}-X\]. 
\begin{lemme}
Si $f$ est positive dans $C_{c}(\D)$ alors $H_{f}^{\alpha} \in \mathcal{S}^{1}$
\end{lemme}
\begin{proof}
De l'égalité, où $T_{f}$ est l'opérateur de Toeplitz:
\[ H^{*}_{f}H_{f}=T_{|f|^{2}}-T_{\bar{f}}T_{f}\]
il suffit de montrer que  pour toute fonction $f$ positive à support compact, on a $T_{f} \in \mathcal{S}^{1/2}$.\\
Notons 
\[ d\mu(w)=f(z)(1-|w|^{2})^{\alpha}dA(w).\]
Comme le disque hyperbolique $D(z,r)$ est un disque euclidien de centre et rayon (p 59 \cite{Z}):
\[ C=\frac{1-s^{2}}{1-s^{2}|z|^{2}}z, \quad R=\frac{1-|z|^{2}}{1-s^{2}|z|^{2}}s, \quad s= \tanh(r) \]
la fonction
\[z \mapsto \mu(D(z,r))=\int_{D(z,r)} d\mu(w) \] est à support compact et donc est dans $L^{1/2}(\D, d\lambda)$. Par le théorème 3 de \cite{ZT}, $T_{f} \in \mathcal{S}^{1/2}$.

\end{proof}
Ce dernier lemme et l'égalité (\ref{eq1}) permet d'écrire pour $f$ positive dans $C_{c}(\D)$:
\[Tr_{\omega}((I - M_{\alpha}C_{\alpha}-X)M_{\bar{f}})=0.\]
Si à présent $||f-g||_{\infty}<\epsilon$, on a
\begin{eqnarray*}
|Tr_{\omega}(TM_{\bar{g}})-Tr_{\omega}(TM_{\bar{f}})|&=& |Tr_{\omega}(TM_{g-f})|\\
&\leq& |Tr_{\omega}(T)|||M_{g-f}||\\
&\leq& |Tr_{\omega}(T)|||f-g||_{\infty}
\end{eqnarray*}
et donc
\[Tr_{\omega}(TM_{\bar{g}})=0\]
Mais $M_{g}$ est inversible borné sur $L^{2}$, on a donc
\[Tr_{\omega}(T)=0\]
Ainsi pour $f \in H'$ et tout $n$ entier:
\[s_{n}(H_{\bar{f}}^{\alpha}) \leq s_{n}(M_{\alpha}H_{\bar{f}}C_{\alpha})+s_{n}(TM_{f})\] 
et donc
\[ Tr_{\omega}( |H_{\bar{f}}^{\alpha}|) =   Tr_{\omega}(M_{\alpha}|H_{\bar{f}}|C_{\alpha}) \leq ||M_{\alpha}||Tr_{\omega}( |H_{\bar{f}}|). \]
De la même façon:
\[||M_{\alpha}^{-1}||Tr_{\omega}( |H_{\bar{f}}|^{\alpha}) \geq Tr_{\omega}( M_{\alpha}^{-1}|H_{\bar{f}}^{\alpha}|C_{\alpha}^{-1})=   Tr_{\omega}(H_{\bar{f}}).\] 
Comme $M_{\alpha}^{-1}$ est un opérateur diagonal on a $||M_{\alpha}^{-1}||=||M_{\alpha}||^{-1}$ et donc
\[Tr_{\omega}( |H_{\bar{f}}^{\alpha}|) = ||M_{\alpha}|| Tr_{\omega}( |H_{\bar{f}}|). \]
Par l'égalité (\ref{casmonome}), on a $||M_{\alpha}|| = \sqrt{\alpha+1}$ et donc:
\[ Tr_{\omega}( |H_{\bar{f}}^{\alpha}|) = \sqrt{\alpha+1} \int_{\T} |f'| d\theta. \]

Les arguments  précédents restent valable pour une symbole borné de $L^{2,\alpha}$, ce qui prouve le théorème \ref{théorème2}.\\

\section {Généralisation}

De façon plus générale, pour $k>1$ un réel et $f \in \mathcal{A}^{2}$, on peut se demander quand est ce que $|H_{\bar{f}}|^{k}$ est dans l'idéal de Macaev, c'est à dire quand est ce que $f \in \mathcal{D}^{k}$. On peut montrer que l'espace de Besov $\mathcal{B}^{p}$ est l'ensemble des $f \in \mathcal{H}ol(\D)$ vérifiant
\begin{eqnarray}\label{besovfraction}
(1-|z|^{2})^{t}R^{t}f(z) \in L^{p}(d\lambda), \quad pt>1 
\end{eqnarray}
où $R^{t}f$ est la dérivée fractionnaire de $f$ (voir ci dessous). Avec la même approche que dans la section $3$, on peut alors penser que $\mathcal{D}^{k}$ est l'ensemble des fonctions vérifiant:
\[ \int_{\T} |R^{1/k}f|^{k} d\theta < \infty. \]
Ce n'est pas le cas, $\mathcal{D}^{k}$ contient strictement cet espace. La situation est ici différente du cas $k=1$ puisqu'il existe des symboles non constants vérifiant \[|H_{\bar{f}}|^{k} \in \mathcal{S}^{1}. \]  

Soit $t$ un réel positif et $f$ une fonction holomorphe sur $\D$ avec
\[ f(z)=\sum_{k=0}^{\infty} a_{k}z^{k} \]
on définit (voir p 18 \cite{Z1})
\[ R^{t}f(z):=  \sum_{k=1}^{\infty} k^{t}a_{k}z^{k} \]
et
\[ R^{0,t}f(z):=  \sum_{k=0}^{\infty} \frac{\Gamma(2)}{\Gamma(2+t)}\frac{\Gamma(2+t+k)}{\Gamma(2+k)} a_{k}z^{k} \]
Pour $t$ fixé, par Stirling, on a:
\begin{eqnarray}\label{equivalentdérivée}
 \frac{\Gamma(2+t+k)}{\Gamma(2+k)} \sim k^{t}
\end{eqnarray}

\begin{lemme}
Pour $f \in \mathcal{A}^{2}$, on a:
\[ f(z)=(t+1)\int_{\D} \frac{(1-|w|^{2})^{t}R^{0,t}f(w)}{(1-z\bar{w})^{2}}dA(w). \]
\end{lemme}
\begin{proof}
Soit $g$ la fonction définie par
\[ g(z)=(t+1)\int_{\D} \frac{(1-|w|^{2})^{t}R^{0,t}f(w)}{(1-z\bar{w})^{2}}dA(w). \]
Par le théorème 2.19 p 54 de \cite{Z1}, cette expression a bien un sens.
Dérivons sous l'intégrale:
\[ R^{0,t}g(z)=(t+1)\int_{\D} (1-|w|^{2})^{t}R^{0,t}f(w)R^{0,t}\left(\frac{1}{(1-z\bar{w})^{2}}\right)dA(w) \]
Par la proposition 1.14 p 19 de \cite{Z1}:
\[ R^{0,t}\left(\frac{1}{(1-z\bar{w})^{2}}\right) = \frac{1}{(1-z\bar{w})^{2+t}} \]
et donc
\begin{eqnarray*}
 R^{0,t}g(z)&=&(t+1)\int_{\D} (1-|w|^{2})^{t}R^{0,t}f(w)\frac{1}{(1-z\bar{w})^{2+t}}dA(w) \\
&=& R^{0,t}f(z)
\end{eqnarray*}
puisque $R^{0,t}f \in L^{1}((1-|z|^{2})^{t}dA)$, voir p 53 de \cite{Z}.\\
Ainsi \[R^{0,t}f= R^{0,t}g\] et l'on a terminé.
\end{proof}
On peut alors comparer les deux normes de Besov (\ref{besovclassique}) et (\ref{besovfraction}):
\begin{lemme}
Il existe $C>0$ (dépendant uniquement de $k$) vérifiant pour tout $t>0$, $p>1$ et tout $f \in \mathcal{A}^{2}$:
 \[ ||(1-|z|^{2})f'(z)||_{L^{pk}(d\lambda)} \leq C ||(1-|z|^{2})^{1/k}R^{1/k}f(z) ||_{L^{pk}(d\lambda)} \]
\end{lemme}
\begin{proof}
Le lemme précédent affirme que pour tout $t>0$
\[ f(z)=(t+1)\int_{\D} \frac{(1-|w|^{2})^{t}R^{0,t}f(w)}{(1-z\bar{w})^{2}}dA(w) \]
ainsi en dérivant sous le signe intégral:
\[ (1-|z|^{2})f'(z)=(1-|z|^{2})(t+1)\int_{\D} 2\bar{w}\frac{(1-|w|^{2})^{t}R^{0,t}f(w)}{(1-z\bar{w})^{3}}dA(w). \]
Comme $pk>k>1$, par le lemme 5.3.2 p 90 de \cite{Z}, il existe $C$ dépendant uniquement de $k$ telle que:
  \[||(1-|z|^{2})f'(z)||_{L^{pk}(d\lambda)} \leq C ||(1-|z|^{2})^{t}R^{0,t}f(z) ||_{L^{pk}(d\lambda)}\]
	On conclut avec (\ref{equivalentdérivée}) et en posant $t=1/k$.
\end{proof}
Nous pouvons à présent prouver l'inclusion annoncée
\begin{proposition}
Soit $k>1$, si $f \in \mathcal{A}^{2}$ vérifie
\[ \int_{\T} |R^{1/k}f|^{k} d\theta < \infty \]
alors 
\[ |H_{\bar{f}}|^{k} \in \mathcal{S}^{1}_{+} \] et
\[ Tr_{\omega}(|H_{\bar{f}}|^{k}) \lesssim \int_{\T} |R^{1/k}f|^{k} d\theta \]
\end{proposition}

\begin{proof}
Comme $k>1$, on a uniformément en $p$:
\[ |||H_{\bar{f}}|^{k}||_{\mathcal{S}^{p}}^{p} = |||H_{\bar{f}}|^{k}||_{\mathcal{S}^{pk}}^{pk} \simeq ||f||_{\mathcal{B}^{pk}}^{pk}\]
où
\[||f||_{\mathcal{B}^{pk}}^{pk} = ||(1-|z|^{2})f'(z)||_{L^{pk}(d\lambda)}^{pk}. \] 
Ainsi, par le lemme précédent, uniformément en $p>1$:
\begin{eqnarray*}
|||H_{\bar{f}}|^{k}||_{\mathcal{S}^{p}}^{p} &\lesssim& ||(1-|z|^{2})f'(z)||_{L^{pk}(d\lambda)}^{pk} \\
&\lesssim& ||(1-|z|^{2})^{1/k}R^{1/k}f(z) ||_{L^{pk}(d\lambda)}^{pk} \\
&=& \int_{\D} (1-|z|^{2})^{p-2} |R^{1/k}f(z)|^{pk} dA(z)
\end{eqnarray*}
et donc, par un raisonnement analogue à celui de la section 3:
\begin{eqnarray*}
|||H_{\bar{f}}|^{k}||_{\mathcal{S}^{1}_{+}} &=& \sup_{p\in ]1;2]} (p-1)|||H_{\bar{f}}|^{k}||_{\mathcal{S}^{p}}\\
&\lesssim& \sup_{p\in ]1;2]} (p-1)\left(\int_{\D} (1-|z|^{2})^{p-2} |R^{1/k}f(z)|^{pk} dA(z)\right)^{1/k}\\
&\lesssim& \left( \int_{\T} |R^{1/k}f|^{k} d\theta \right)^{1/k}
\end{eqnarray*}
Pour la majoration de la trace, on reprend la méthode de la section 3, comme $k>1$, la fonction $|R^{1/k}f|^{k}$ est sous harmonique et on a:
\[ \int_{\T} |R^{1/k}f|^{k} d\theta = \lim_{p\rightarrow 1} (p-1) \int_{\D} (1-|z|^{2})^{p-2} |R^{1/k}f(z)|^{pk} dA(z)\]
On conclut grâce à (\ref{connes1})
\end{proof}
Il reste à montrer que l'inclusion est stricte. Soit 
\[ f(z)	= \sum_{n=0}^{\infty} z^{n} \]
on a uniformément en $p$:
\begin{eqnarray*}
(p-1)||f||_{\mathcal{B}^{pk}} &=& (p-1) \left( \int_{\D} \left|\left(\sum_{k} z^{n}\right)'\right|^{pk} (1-|z|^{2})^{pk-2} dA(z)\right)^{1/pk}\\
&\leq& (p-1) \sum_{n} \left(\int_{\D} \left|( z^{n})'\right|^{pk} (1-|z|^{2})^{pk-2} dA(z) \right)^{1/pk}\\
&\simeq&(p-1) \left(\sum_{n} \int_{0}^{1} (nr)^{(n-1)pk}(1-r)^{pk-2}rdr\right)^{1/pk}\\
&=& (p-1) \left(\sum_{n} n^{(n-1)pk} \beta((n-1)pk+2,pk-1)\right)^{1/pk}\\
&\simeq&(p-1) \left(\sum_{n} n^{1-p}\right)^{1/pk}\\
\end{eqnarray*}
où $\beta$ est la fonction Bêta. Si l'on pose
\[F(z)= \sum_{n=0}^{\infty} z^{\lambda_{n}}\] une série est lacunaire au sens d'Hadamard, c'est à dire avec $\frac{\lambda_{k+1}}{\lambda_{k}}\geq c >1$, alors $c^{k} \lesssim \lambda_{k}$ et on a pour tout $p\in ]1;2]$:
\begin{eqnarray*}
(p-1)||F||_{\mathcal{B}^{pk}}^{pk} &\lesssim& (p-1) \sum_{n} c^{-n(p-1)}\\
&=& \frac{(p-1)}{c^{p-1}-1}\\
&\lesssim& (\ln(c))^{-1}.\\
\end{eqnarray*}
Passant au sup, $F \in \mathcal{D}^{k}$. Mais
\[R^{1/k}F(z)= \sum_{n=0}^{\infty} \lambda_{n}^{k}z^{\lambda_{n}}\] n'est pas dans $H^{1}$ par un théorème de Paley p 104 \cite{Du}.\\
\newline
\textbf{Remerciements:} je tiens à remercier chaleureusement M. Englis pour ses remarques et ses conseils.

\bibliographystyle{plain}
\bibliography{biblio}

\begin{thebibliography}{10}

\bibitem{AFP85}
J.~Arazy, S.~Fisher, and J.~Peetre.
\newblock Mobius invariant function spaces.
\newblock {\em J. reine angew. Math.}, 363:110--145, 1986.

\bibitem{AFP}
J.~Arazy, S.~Fisher, and J.~Peetre.
\newblock {Hankel operators on weighted Bergman spaces}.
\newblock {\em Amer. J. Math.}, 110:989--1053, 1988.

\bibitem{A}
S.~Axler.
\newblock {The Bergman space, the Bloch space and commutators of multiplication
  operators}.
\newblock {\em Duke Math. J.}, 53:315--332, 1986.

\bibitem{C}
A.~Connes.
\newblock {\em {Noncommutative Geometry}}.
\newblock Academic Press, San Diego, CA, 1994.

\bibitem{CM}
A.~Connes and H.~Moscovici.
\newblock {The Local Index Formula Index in Noncommutative Geometry}.
\newblock {\em Geom. Funct. Anal.}, 5:174--243, 1995.

\bibitem{Du}
P.L. Duren.
\newblock {\em {Theory of $H^{p}$ spaces}}.
\newblock Pure and Applied Mathematics 38. Academic Press, New York-London,
  1970.

\bibitem{EGZ}
M.~Englis, K.~Guo, and G.~Zhang.
\newblock {Toeplitz and Hankel operators and Dixmier traces on the unit ball of
  $\C^{n}$}.
\newblock {\em Proc. Amer. Math Soc. in press}, 137:3669--3678, 2009.

\bibitem{ER}
M.~Englis and R.~Rochberg.
\newblock {The Dixmier trace of Hankel operators on the Bergman space}.
\newblock {\em J. Funct. Anal.}, 257:1445--1479, 2009.

\bibitem{GK}
I.~Gohberg and M.~G. Krein.
\newblock {\em {Introduction to the theory of linear nonselfadjoint
  operators}}.
\newblock Transl. Math. Monographs, 1969.

\bibitem{LR}
S.Y Li and B.~Russo.
\newblock {Hankel operators in the Dixmier class}.
\newblock {\em C. R. Acad. Sci. Paris}, 325:21--26, 1997.

\bibitem{L}
D.~Luecking.
\newblock {Characterizations of certain classes of Hankel operators on the
  Bergman spaces of the unit disk}.
\newblock {\em J. Func. Anal.}, 110:247--271, 1992.

\bibitem{MV}
R.~Meise and D.~Vogt.
\newblock {\em {Introduction to Functional Analysis}}.
\newblock Oxford Sci. Pub., 1997.

\bibitem{Pa}
M.~Pavlovic.
\newblock {\em {Introduction to function spaces on the disk}}.
\newblock Posebna Izdanja 20. Matematicki Institut SANU, Belgrade, 2004.

\bibitem{Pe}
V.~Peller.
\newblock {\em {Hankel Operators and their applications}}.
\newblock Springer-Verlag, New-York, 2003.

\bibitem{RuARC}
W.~Rudin.
\newblock {\em {Real an complex analysis}}.
\newblock McGraw-Hill Book Company, New York, 1966.

\bibitem{SY}
K.~Seip and E.H Youssfi.
\newblock {Hankel Operators on Fock Spaces and Related Bergman Kernel
  Estimates}.
\newblock {\em J. Geo. Anal.}, Online First, 2011.

\bibitem{S}
B.~Simon.
\newblock {\em {Trace ideal and their applications}}.
\newblock Cambridge University Press, 1979.

\bibitem{T}
R.~Tytgat.
\newblock {Classe de Dixmier d'opérateurs de Hankel}.
\newblock {\em Journal of Operator Theory}, à paraitre, 201?

\bibitem{Z}
K.~Zhu.
\newblock {\em {Operator Theory in Function Spaces}}.
\newblock Marcel Dekker, New York, 1990.

\bibitem{ZABS}
K.~Zhu.
\newblock {Analytic Besov spaces}.
\newblock {\em Journal of Mathematical analysis and applications},
  157:318--336, 1991.

\bibitem{Z1}
K.~Zhu.
\newblock {\em {Spaces of Holomorphic Functions in the Unit Ball}}.
\newblock Graduate Texts in Mathematics 226, Springer-Verlag, New York, 2005.

\bibitem{ZT}
K.~Zhu.
\newblock {Schatten class Toeplitz operators on weighted Bergman spaces of the
  unit ball}.
\newblock {\em New York J. Math.}, 13:299--316, 2007.

\end{thebibliography}

 \end{document}